\documentclass[11pt]{article}

\usepackage[paperwidth=8.5in,paperheight=11in,portrait,top=1in,bottom=1.in,left=1.15in,right=1.15in]{geometry}
\usepackage[utf8]{inputenc}
\usepackage{authblk}
\usepackage{blkarray}

\usepackage{amsfonts,amsmath,amssymb,amsthm}
\usepackage{latexsym,mathrsfs,mathtools,bm}
\usepackage{graphicx,subcaption,epsfig,caption,subcaption,float,xcolor}
\usepackage{enumitem}
\usepackage{chngcntr}

\usepackage{tikz}
\usetikzlibrary{calc}

\usepackage[hidelinks]{hyperref}
\usepackage{bookmark}

\theoremstyle{plain}
\newtheorem{thm}{Theorem}[section]

\newtheorem{lem}[thm]{Lemma}
\newtheorem{prop}[thm]{Proposition}

\newtheorem{defi}[thm]{Definition}
\newtheorem{rem}[thm]{Remark}

\numberwithin{equation}{section}
\newcommand{\cP}{\mathcal{P}}

\newcommand{\wcA}{\widehat{\mathcal{A}}}
\newcommand{\wcB}{\widehat{\mathcal{B}}}
\newcommand{\wcC}{\widehat{\mathcal{C}}}

\newcommand{\cD}{\mathcal{D}}

\newcommand{\cZ}{\mathcal{Z}}

\newcommand{\II}{\mathbb{I}}
\newcommand{\JJ}{\mathbb{J}}
\newcommand{\NN}{\mathbb{N}}

\newcommand{\bv}{\mathbf{v}}
\newcommand{\bx}{\mathbf{x}}
\newcommand{\by}{\mathbf{y}}

\numberwithin{equation}{section}
\title{\bf A bivariate $Q$-polynomial structure for\\ the non-binary Johnson scheme }

\renewcommand*{\Affilfont}{\normalsize\small}

\author[1]{Nicolas Cramp\'e}
\author[2]{Luc Vinet}
\author[3]{Meri Zaimi}
\author[4]{Xiaohong Zhang\vspace{.5em}}
\affil[1]{\textit{Institut Denis-Poisson CNRS/UMR 7013 - Université de Tours - Université
d'Orléans, \newline\vspace{.9em}
Parc de Grandmont, 37200 Tours, France.}}
\affil[2,3,4]{\textit{Centre de Recherches Math\'ematiques, Universit\'e de Montr\'eal,
\newline\vspace{.9em}
P.O. Box 6128, Centre-ville Station, Montr\'eal (Qu\'ebec), H3C 3J7, Canada.}}
\affil[2]{\textit{IVADO, Montr\'eal (Qu\'ebec), H2S 3H1, Canada. \newline\vspace{.9em}}}

{
 \makeatletter
 \renewcommand\AB@affilsepx{: \protect\Affilfont}
 \makeatother
 \affil[ ]{E-mail addresses}
 \makeatletter
 \renewcommand\AB@affilsepx{, \protect\Affilfont}
 \makeatother
 \affil[1]{crampe1977@gmail.com}
  \affil[2]{luc.vinet@umontreal.ca}
  \affil[3]{meri.zaimi@umontreal.ca}
  \affil[4]{xiaohong.zhang@umontreal.ca}
}
\date{\today}

\begin{document}
\maketitle

%\hrule
\begin{abstract}
The notion of multivariate $P$- and $Q$-polynomial association scheme has been introduced recently, generalizing the well-known
univariate case. Numerous examples of such association schemes
have already been exhibited. In particular, it has been demonstrated that the non-binary Johnson scheme is a bivariate $P$-polynomial association scheme. 
We show here that it is also a bivariate $Q$-polynomial association scheme for some parameters. This provides, with the $P$-polynomial structure, the bispectral property (\textit{i.e.}\ the recurrence and difference relations) of a family of bivariate orthogonal polynomials made out of univariate Krawtchouk and dual Hahn polynomials. The algebra based on the bispectral operators is also studied together with the subconstituent algebra of this association scheme.
\end{abstract}
%\hrule

\section{Introduction}

The main purpose of this paper is to provide a bivariate $Q$-polynomial structure for the non-binary Johnson scheme.

Association schemes are basic objects in algebraic combinatorics which arise in the study of various topics such as coding, design and group theories \cite{BI,God2,Del,BM,Bai,Zie2}. An important class of association schemes are those which
are said to be $P$- and/or $Q$-polynomial. These association schemes have a structure involving orthogonal polynomials. For association schemes which are both $P$- and $Q$-polynomial,
the underlying orthogonal polynomials are bispectral (\textit{i.e.}\ they satisfy a differential or difference equation in addition to their recurrence relation) and belong to the Askey scheme \cite{Leo,BI,KLS}. Moreover, the algebraic
structures that can be defined from $P$- and $Q$-polynomial association schemes are closely related to the theory of tridiagonal pairs, Leonard pairs, and to the Askey--Wilson algebra \cite{ITT,Ter01,Zhedanov}.

The concepts of $P$-polynomial and $Q$-polynomial association schemes have recently been extended to the bivariate case in \cite{BCPVZ} and to the multivariate case in \cite{BKZZ}. In either case, 
similar characterizations of such schemes in terms of the recurrence relation of certain bivariate or multivariate polynomials, and in terms of constraints on intersection numbers or Krein parameters are given, as for the univariate case.  
This provides a framework for studying the polynomial structures of higher rank association schemes.
Several examples of association schemes have been shown to admit a bivariate or multivariate polynomial structure in \cite{BCPVZ,BKZZ}: direct product of association schemes, the symmetrization of association schemes (also called extension), the non-binary Johnson scheme and generalized Johnson schemes, association schemes based on isotropic or attenuated spaces, and others.
%Several examples of association schemes have been shown to admit a bivariate or multivariate polynomial structure in \cite{BCPVZ,BKZZ}: direct product of association schemes, the symmetrization of association schemes (also called extension), the $24$-cell, the non-binary Johnson scheme and generalized Johnson schemes, association schemes based on isotropic or attenuated spaces, dodecahedron, composition of Gelfand pairs. 
For some of these examples, such as the non-binary Johnson scheme, only a bivariate or multivariate $P$-polynomial structure was obtained. In this paper, we focus on the non-binary Johnson scheme $J_r(k,n)$ and show that, at least for $n\geq 2k-1$, it admits a bivariate $Q$-polynomial structure with respect to the definitions introduced in \cite{BCPVZ}. This allows to obtain the bispectrality of the associated polynomials. We also examine some algebras related to this example of bivariate $P$- and $Q$-polynomial association scheme,  namely the algebra of bispectrality and the subconstituent algebra.       

The paper is organized as follows. In Section \ref{sec:BiPQAS}, the notion of association scheme and the bivariate $P$- and $Q$-polynomial properties are recalled. In Section \ref{sec:nbJAS}, the definition of the non-binary Johnson scheme and some of its important features are reviewed. In Section \ref{sec:biQpoly}, a bivariate $Q$-polynomial structure for the non-binary Johnson scheme is provided through the study of the recurrence properties of its dual eigenvalues. In Section \ref{sec:bispectrality}, the bispectrality of the bivariate polynomials associated to the non-binary Johnson scheme is discussed. Moreover, the algebra of the associated bispectral operators as well as the subconstituent algebra of the non-binary Johnson scheme are explored and connected. Some relevant properties of hypergeometric polynomials are recalled in Appendix \ref{app:A}.    

\section{Bivariate $P$- and $Q$-polynomial association schemes} \label{sec:BiPQAS}

\subsection{Association scheme}

Let us recall the definition of an association scheme (see \cite{BI,BIT} for more details). 
The set $\cZ=\{A_0,\dots,A_N\}$ is a symmetric commutative association scheme with $N$ classes if 
the matrices $A_i$, called adjacency matrices, are non-zero $\bv\times \bv$ matrices with $0$ and $1$ entries
satisfying:
\begin{itemize}
 \item[(i)] $A_0=\II$, where $\II$ is the $\bv\times \bv$ identity matrix;
 \item[(ii)] $\displaystyle \sum_{i=0}^N A_i=\JJ$,  where $\JJ$ is the $\bv\times \bv$ matrix filled with $1$;
 \item[(iii)] $A_i^t=A_i$ for $i=0,1,\dots N$ and $.^t$ stands for the transposition;
 \item[(iv)] The following relations hold
 \begin{equation}
  A_iA_j=A_jA_i=\sum_{k=0}^N p_{ij}^k A_k,
 \end{equation}
 where $p_{ij}^k$ are constants called intersection numbers.
\end{itemize}
The (commutative) algebra generated by the $A_i$'s is called the Bose-Mesner algebra of the association scheme. 

\subsection{Monomial orders}

The total order \textit{deg-lex} on the monomials, denoted $\leq$, is defined by 
\begin{equation}
 x^my^n \leq x^iy^j \Leftrightarrow \begin{cases}
                                         m+n<i+j \\
                                         \text{or}\\
                                        m+n=i+j \quad\text{and}\quad n \leq j\,.
                                        \end{cases} \label{eq:ord}
\end{equation}
The degree, associated to the total order \textit{deg-lex}, of a polynomial $v(x,y)$ in two variables $x$ and $y$ is the couple $(i,j)$ such that $x^iy^j$ is the greatest monomial in $v(x,y)$.

We refine this definition with the following partial order on monomials:
\begin{equation}
 x^my^n \preceq_{(\alpha,\beta)}  x^i y^j \Leftrightarrow \begin{cases}
                                        m+\alpha n\leq i+\alpha j \\
                                         \text{and}\\
                                     \beta m+ n\leq \beta i+ j\,,
                                        \end{cases} \label{eq:partord}
\end{equation}
where  $0 \leq \alpha\leq 1$ and $0 \leq \beta < 1$. This also defines a partial order on $\NN^2$.

This leads to the following two definitions for bivariate polynomials and subsets of $\NN^2$.
A bivariate polynomial $v(x,y)$ is called $(\alpha,\beta)$-compatible of degree $(i,j)$ if the monomial $x^iy^j$ appears and all other monomials $x^my^n$ appearing are smaller than $x^iy^j$ for the order $\preceq_{(\alpha,\beta)}$.
% A bivariate polynomial of degree $(i,j)$ with all its monomials smaller than $(i,j)$ for the order $\preceq_{(\alpha,\beta)}$ is called $(\alpha,\beta)$-compatible.
A subset $\cD$ of $\NN^2$ is called $(\alpha,\beta)$-compatible if for any $(i,j) \in \cD$, one gets 
\begin{equation} 
%\Big( x^my^n\preceq_{(\alpha,\beta)} x^iy^j \Big) \Rightarrow  \Big( (m,n)\in \cD \Big).
\Big( (m,n)\preceq_{(\alpha,\beta)} (i,j) \Big) \Rightarrow  \Big( (m,n)\in \cD \Big).
\end{equation}

\subsection{Bivariate $P$- and $Q$-polynomial association scheme}

The notion of $P$-polynomial association scheme has been generalized to the bivariate case as follows in \cite{BCPVZ}.
\begin{defi} \label{def:bi}
Let $\cD \subset \NN^2$, $0 \leq \alpha\leq 1$, $0 \leq \beta<1$ and $\preceq_{(\alpha,\beta)}$ be the order defined in \eqref{eq:partord}.
 The association scheme $\cZ=\{A_0,\dots,A_N\}$ is called a bivariate $P$-polynomial association scheme of type $(\alpha,\beta)$ on the domain $\cD$ if these two conditions are satisfied:
 \begin{itemize}
  \item[(i)] there exists a relabeling of the adjacency matrices:
 \begin{equation}
  \{A_0,A_1,\dots, A_N\} = \{ A_{mn} \ |\ (m,n) \in \cD \},
 \end{equation}
such that, for $(i,j) \in \cD$,
\begin{equation}
 A_{ij}=v_{ij}(A_{10},A_{01})\,, \label{eq:vij}
\end{equation}
where  $v_{ij}(x,y)$ is a $(\alpha,\beta)$-compatible bivariate polynomial of degree $(i,j)$;
\item[(ii)] $\cD$ is $(\alpha,\beta)$-compatible.
 \end{itemize}
\end{defi}
Let us remark that the previous definition can also be given for other choices of the orders (see  \cite{BKZZ} for a generalization of this definition). With the relabeling of the adjacency matrices as in Definition \ref{def:bi}, the intersection numbers now read
\begin{equation}
 A_{ij}A_{k\ell}= \sum_{(m,n)\in \cD}  p_{ij,k\ell}^{mn}\, A_{mn} \, .  \label{eq:intern}
\end{equation}

Let $\cZ=\{ A_{mn} \ |\ (m,n) \in \cD \}$ be a bivariate $P$-polynomial association scheme of  type $(\alpha,\beta)$.
Since the matrices $A_{ij}$ are pairwise commuting, they can be diagonalized in the same basis. 
The vector space $V$ of dimension $\bv$, on which the adjacency matrices act, can be decomposed as follows 
\begin{equation}
 V=\bigoplus_{(m,n) \in \cD^\star} V_{mn}\,,
\end{equation}
where $\cD^\star$ is a subset of $\NN^2$ with the same cardinality as $\cD$ and $V_{mn}$ is a common eigenspace for all the matrices $A_{ij}$.
Let $E_{mn}$ with $(m,n) \in \cD^\star$ denote the projector on the corresponding eigenspace: $E_{mn}V=V_{mn}$. They satisfy
\begin{eqnarray}
 &&E_{mn}E_{pq}=\delta_{mn,pq} E_{mn}\ , \qquad \sum_{(m,n) \in \cD^\star} E_{mn}=\II \ , \qquad E_{00}=\frac{1}{\bv} \JJ,\\
 &&  A_{ij} =\sum_{(m,n) \in \cD^\star} p_{ij}(mn) E_{mn}\,, \label{eq:AE}
\end{eqnarray}
with $p_{ij}(mn) $ the eigenvalues of $A_{ij}$ in the subspace $V_{mn}$. The idempotents $E_{mn}$ also generate the Bose--Mesner algebra.
One gets a relation between these eigenvelues  and the polynomials $v_{ij}(x,y)$:
\begin{equation}
  p_{ij}(mn)=v_{ij}(\theta_{mn},  \mu_{mn})\,, \label{eq:pv}
 \end{equation}
where $\theta_{mn}=p_{10}(mn)$ and $\mu_{mn}=p_{01}(mn)$ are the eigenvalues of $A_{10}$ and $A_{01}$ on $E_{mn}$, respectively, and $v_{ij}$ is as in Definition~\ref{def:bi} ($A_{ij}=v_{ij}(A_{10},A_{01})$).
On the other hand, if an association scheme $\{A_{ij} \ | \  (i,j)\in \cD \} $ on an $(\alpha,\beta)$-compatible region $\cD$  has eigenvalues
satisfying \eqref{eq:pv} for some $(\alpha,\beta)$-compatible bivariate polynomial $v_{ij}(x,y)$  of degree $(i,j)$, 
then this scheme is a bivariate $P$-polynomial association scheme of type $(\alpha,\beta)$.

Relation \eqref{eq:AE} can be inverted and one gets
\begin{eqnarray}
E_{mn} =\frac{1}{\bv}\sum_{(i,j) \in \cD} q_{mn}(ij) A_{ij}\ . \label{eq:EA}
\end{eqnarray}
The parameters $q_{mn}(ij)$, called dual eigenvalues, are related to the eigenvalues $p_{ij}(mn) $ by the Wilson duality \cite{BI}:
\begin{equation}
 \frac{q_{mn}(ij) }{m_{mn}}=\frac{p_{ij}(mn) }{k_{ij}}\,, \label{eq:qp}
\end{equation}
where 
\begin{equation}
 k_{ij}=p_{ij}(00), \qquad m_{ij}=q_{ij}(00)\,, \label{eq:km}
\end{equation}
 are the valence and the multiplicity, respectively.
This relation holds for any symmetric association scheme and the proof can be found in \cite[Theorem 3.5]{BI} for example. 

The idempotents $E_{ij}$ of an association scheme also satisfy a relation dual to \eqref{eq:intern} given by 
\begin{equation}
 E_{ij} \circ E_{k\ell} = \frac{1}{\bv}\sum_{(m,n)\in\cD^\star} q_{ij,k\ell}^{mn} E_{mn}\,,
\end{equation}
where $\circ$ is the Hadamard product (or entrywise product). The numbers $q_{ij,k\ell}^{mn}$ are called Krein parameters.

The notion of $Q$-polynomial association scheme is developed in \cite{Del} (see also \cite{BCN,TerCourse}) and is dual to the $P$-polynomial one.
A generalization to bivariate $Q$-polynomial association scheme is given below \cite{BCPVZ}.
\begin{defi} \label{def:biQ}
Let $\cD^\star \subset \NN^2$, $0 \leq \alpha\leq 1$, $0 \leq \beta<1$ and $\preceq_{(\alpha,\beta)}$ be the order \eqref{eq:partord}.
 The association scheme with idempotents $E_0,E_1,\dots E_N$ is called a bivariate $Q$-polynomial of type $(\alpha,\beta)$ on the domain $\cD^\star$ if these two conditions are satisfied:
\begin{itemize}
  \item[(i)]there exists a relabeling of the idempotents:
 \begin{equation}
  \{E_0,E_1,\dots E_N\} = \{ E_{mn} \ |\ (m,n) \in \cD^\star \}\,,
 \end{equation}
 such that, for $(m,n) \in \cD^\star$,
\begin{equation}
 \bv\, E_{mn}=v^\star_{mn}( \bv\,E_{10}, \bv\, E_{01})\quad \text{(under the Hadamard product)},
\end{equation}
 where  $v^\star_{mn}(x,y)$ is a $(\alpha,\beta)$-compatible bivariate polynomial of degree $(m,n)$;
\item[(ii)] $\cD^\star$ is $(\alpha,\beta)$-compatible.
\end{itemize}
\end{defi}

In this paper, the following result will be useful for proving the bivariate $Q$-polynomial property of an association scheme.
\begin{prop}\cite{BCPVZ}\label{pro:q1} Let $\cZ$ be a symmetric association scheme with idempotents $E_{ij}$, for $(i,j)\in \cD^\star \subset \NN^2$. The following items are equivalent:
\begin{itemize}
 \item[(i)] $\cZ$ is a bivariate $Q$-polynomial association scheme of type $(\alpha,\beta)$ on $\cD^\star$;
  \item[(ii)] $\cD^\star$ is $(\alpha,\beta)$-compatible and the Krein parameters satisfy, for  $(i,j),(i+1,j) \in \cD^\star$,
  \begin{eqnarray}
   && q_{10,ij}^{i+ 1 , j}\neq 0,\ \  q_{10,i+1j}^{i , j}\neq 0\,, \label{eqq01} \\
   &&\left( q_{10,ij}^{mn}\neq 0\right)\quad \Rightarrow \quad   (m,n)\preceq_{(\alpha,\beta)} (i+1,j)\ \ \text{and} \ \ (i,j)\preceq_{(\alpha,\beta)} (m+1,n)\,,\label{eqq03}
  \end{eqnarray}
  and, for  $(i,j),(i,j+1) \in \cD^\star$,
    \begin{eqnarray}
   && q_{01,ij}^{i, j+1}\neq 0,\ \ q_{01,ij+1}^{i , j}\neq 0\,,  \label{eqq04}\\
 &&\left( q_{01,ij}^{mn}\neq 0\right)\quad \Rightarrow \quad   (m,n)\preceq_{(\alpha,\beta)} (i,j+1)\ \ \text{and} \ \ (i,j)\preceq_{(\alpha,\beta)} (m,n+1)\,; \label{eqq02}
  \end{eqnarray}
    \item[(iii)] $\cD^\star$ is $(\alpha,\beta)$-compatible and the dual eigenvalues $q_{ij}(mn) $ defined by \eqref{eq:EA} satisfy
 \begin{equation}
  q_{ij}(mn)=v^\star_{ij}( \theta^\star_{mn},  \mu^\star_{mn})\,, \label{eq:pv2}
 \end{equation}
 where $\theta^\star_{mn}=q_{10}(mn)$ and $\mu^\star_{mn}=q_{01}(mn)$, and $v^\star_{ij}(x,y)$ is a $(\alpha,\beta)$-compatible bivariate polynomial  of degree $(i,j)$.
\end{itemize}
\end{prop}
Item (iii) of Proposition \ref{pro:q1} shows that we can use the dual eigenvalues of an association scheme to prove its bivariate $Q$-polynomial property, and item (ii) justifies that we can study the recurrence relations of the polynomials $v_{ij}^\star$ expressing the dual eigenvalues.

In this paper, the two relevant types of bivariate association schemes will be $(\alpha,\beta) = (1,0)$ for the $P$-polynomial property and $(\alpha,\beta) = (0,\frac{1}{2})$ for the $Q$-polynomial property. Figure \ref{fig:recu} illustrates some constraints on the nonzero intersection numbers or Krein parameters for these types of compatibility.

\begin{figure}[hbtp]
\begin{center}
\begin{subfigure}[b]{0.2\textwidth}
    \centering
    \begin{tikzpicture}[scale=0.4]
\draw[->] (0,0)--(7,0);\draw[->] (0,0)--(0,7);
\draw [fill] (2,4) circle (0.07);
\draw [fill] (3,4) circle (0.07);
\draw [fill] (1,4) circle (0.07);
\draw (0,7) node[above] {$n$};
\draw (7,0) node[above] {$m$};
\draw (2,0) node[below] {$i$};
\draw (0,4) node[left] {$j$};
\draw[dashed,thin] (2,0)--(2,4) --(0,4);
\end{tikzpicture}
\caption{}
\label{fig:recp10}
\end{subfigure}%
\hfill
\begin{subfigure}[b]{0.2\textwidth}
    \centering
    \begin{tikzpicture}[scale=0.4]
\draw[->] (0,0)--(7,0);\draw[->] (0,0)--(0,7);
\draw [fill] (2,4) circle (0.07);
\draw [fill] (2,3) circle (0.07);

\draw [fill] (3,4) circle (0.07);

\draw [fill] (1,4) circle (0.07);

\draw [fill] (2,5) circle (0.07);
\draw [fill] (3,3) circle (0.07);
\draw [fill] (1,5) circle (0.07);
\draw [fill] (4,3) circle (0.07);
\draw [fill] (0,5) circle (0.07);
\draw (0,7) node[above] {$n$};
\draw (7,0) node[above] {$m$};
\draw (2,0) node[below] {$i$};
\draw (0,4) node[left] {$j$};
\draw[dashed,thin] (2,0)--(2,4) --(0,4);
\end{tikzpicture}
\caption{}
\label{fig:recp01}
\end{subfigure}%
\hfill
\begin{subfigure}[b]{0.2\textwidth}
    \centering
    \begin{tikzpicture}[scale=0.4]
\draw[->] (0,0)--(7,0);\draw[->] (0,0)--(0,7);
\draw [fill] (2,4) circle (0.07);
\draw [fill] (1,4) circle (0.07);
\draw [fill] (3,4) circle (0.07);
\draw [fill] (1,5) circle (0.07);
\draw [fill] (3,3) circle (0.07);
\draw (0,7) node[above] {$n$};
\draw (7,0) node[above] {$m$};
\draw (2,0) node[below] {$i$};
\draw (0,4) node[left] {$j$};
\draw[dashed,thin] (2,0)--(2,4) --(0,4);
\end{tikzpicture}
\caption{}
\label{fig:recq10}
\end{subfigure}%
\hfill
\begin{subfigure}[b]{0.2\textwidth}
    \centering
    \begin{tikzpicture}[scale=0.4]
\draw[->] (0,0)--(7,0);\draw[->] (0,0)--(0,7);
\draw [fill] (2,4) circle (0.07);
\draw [fill] (2,5) circle (0.07);
\draw [fill] (2,3) circle (0.07);
\draw (0,7) node[above] {$n$};
\draw (7,0) node[above] {$m$};
\draw (2,0) node[below] {$i$};
\draw (0,4) node[left] {$j$};
\draw[dashed,thin] (2,0)--(2,4) --(0,4);
\end{tikzpicture}
\caption{}
\label{fig:recq01}
\end{subfigure}
\end{center}

\caption{The coordinates $(m,n)$ of the dots in the graphs represent when (a) $p_{10,ij}^{mn}$ and (b) $p_{01,ij}^{mn}$ may be non-zero for $(\alpha,\beta)=(1,0)$, (c) $q_{10,ij}^{mn}$ and (d) $q_{01,ij}^{mn}$ may be non-zero for $(\alpha,\beta)=(0,\frac{1}{2})$. \label{fig:recu}  }
\end{figure}
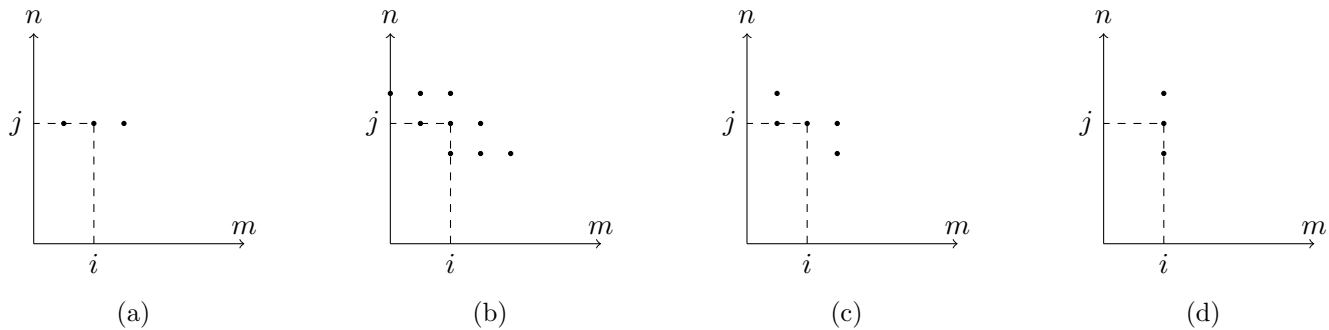

\section{Non-binary Johnson association scheme} \label{sec:nbJAS}

The non-binary Johnson scheme is a generalization of the Johnson scheme whose eigenvalues can be expressed in terms of bivariate polynomials 
formed of Krawtchouk and Hahn polynomials \cite{Dun,TAG}. 

We start by recalling the definition of the non-binary Johnson scheme, which can be found in \cite{TAG}. 
Let $K=\{0,1,2,\dots,r-1\}$, where $r$ is an integer greater than $1$, and consider the $n$-fold Cartesian product $K^n$, where $n$ is a positive integer. 
For a vector $\bx$ in $K^n$ with components $\bx_i$, the weight $w(\bx)$ is defined as the number of non-zero components of $\bx$, that is
\begin{equation}
	w(\bx) = \big|\{ i \ | \ \bx_i \neq 0  \}\big|. \label{eq:weight}
\end{equation}   
For two vectors $\bx,\by \in K^n$, the number of equal non-zero components $e(\bx,\by)$ and the number of common non-zeros $c(\bx,\by)$ are also defined:
\begin{align}
	e(\bx,\by) = \big|\{ i \ | \ \bx_i=\by_i \neq 0  \}\big|, \quad c(\bx,\by) = \big|\{ i \ | \ \bx_i\neq 0, \ \by_i\neq 0 \}\big|. \label{eq:ecnonzeros}
\end{align}
Consider a fixed weight number $k$. Note that we must have $0 \leq k \leq n$ by definition \eqref{eq:weight}. The set
\begin{equation}
	S = \{ \bx \in K^n \ | \ w(\bx)=k \} \,, \label{eq:setXnbJ}
\end{equation} 
together with all the non-empty relations\footnote{The relations $R_{ij}$ of \cite{TAG} have been relabeled as follows: $i \mapsto i+j$ and $j \mapsto j$.}
\begin{equation}
	R_{ij} = \{ (\bx,\by) \in S^2 \ | \ e(\bx,\by)=k-i-j, \ c(\bx,\by)=k-j \}\,, \label{eq:RijNBJ}
\end{equation}    
define a symmetric association scheme called the non-binary Johnson scheme and  denoted $J_r(k,n)$, following the notation of \cite{TAG}. 
From this definition, one can construct the adjacency matrices $A_{ij}$ of the non-binary Johnson scheme. 
These are $|S|\times |S|$ matrices whose entries, labeled by the couples $(\bx,\by)\in S^2$, take the value one if $(\bx,\by) \in R_{ij}$ and zero otherwise.
In particular, for $i=j=0$, the adjacency matrix $A_{00}$ is the identity matrix since $(\bx,\by) \in R_{00}$ if and only if $\bx=\by$. 

When $r=2$, it is seen that $J_2(k,n)$ is the usual Johnson scheme $J(n,k)$ (see \textit{e.g.}\ \cite{BI}) which is a univariate $P$- and $Q$-polynomial association scheme. In what follows, we will suppose $r\geq 3$.

Using only the definition of the scheme, it is possible to compute the following relations between the adjacency matrices:
\begin{align}
	A_{10}A_{ij} =& (k-i-j+1)(r-2) A_{i-1,j} + \left( i(r-3)+j(r-2) \right) A_{ij} +(i+1)A_{i+1,j}, \label{eq:rec10NBJ} \\
	A_{01}A_{ij} =& (k-i-j+1)(r-2)j A_{i-1,j} + (k-i-j+1)(n-k-j+1)(r-1)A_{i,j-1} \nonumber \\ 
	&+ (i+1)j A_{i+1,j} + (j+1)^2 A_{i,j+1} + (i+1)(n-k-j+1)(r-1) A_{i+1,j-1} \nonumber \\
	&+ (j+1)^2(r-2) A_{i-1,j+1} +j\left( k-i-j+(r-2)i+(n-k-j)(r-1)\right)A_{i,j}. \label{eq:rec01NBJ}
\end{align} 

From the relations \eqref{eq:rec10NBJ} and \eqref{eq:rec01NBJ}, we can deduce (see Proposition 2.6 in \cite{BCPVZ}):
\begin{prop}\cite{BCPVZ}
 The non-binary Johnson scheme $J_r(k,n)$ is a bivariate $P$-polynomial association scheme on $\cD=\{(a,b)\in \NN^2\, | \, a+b \leq k, \ b \leq k,n-k\}$ of type $(1,0)$.
\end{prop}
Note that if $k\leq n-k$, then the domain $\cD \subseteq \mathbb{N}^2$ of the couples $(i,j)$ for the adjacency matrices $A_{ij}$ is the triangle $i+j \leq k$, while if $n-k < k$, the domain is the same triangle 
truncated horizontally at $j=n-k$. In both cases, the domain $\cD$ is $(1,0)$-compatible.

An explicit expression of the eigenvalues $p_{ij}(x,y)$ is also known \cite{TAG} and is given in terms of 
bivariate polynomials\footnote{The following change of variables has been applied to the eigenvalues given in \cite{TAG} in order to fit with our conventions: $i \mapsto i+j, \ x\mapsto x+y, \ y\mapsto x$. The definitions of the polynomials have also been changed: $K_i(N,p,x)\mapsto K_i(x,N,p)$ and $E_i(N,p,x)\mapsto E_i(x,N,p)$.}:
\begin{equation}
	p_{ij}(x,y) = (r-1)^jK_{i}(x,k-j,r-1)E_{j}(y,n-x,k-x), \label{eq:pijNBJ}
\end{equation}
where $(i,j),(x,y) \in \{(a,b)\in \NN^2\, | \, a+b \leq k, \ b \leq k,n-k\}$ and for $i=0,1,\dots,N$,
\begin{align}
	&K_i(x,N,p) = \sum_{\ell=0}^i(-1)^\ell (p-1)^{i-\ell} {x\choose \ell} {N-x\choose i-\ell}, \label{eq:krawtchouk} \\
	&E_i(x,N,p) = \sum_{\ell=0}^i (-1)^\ell {x\choose \ell} {p-x\choose i-\ell} {N-p-x\choose i-\ell}
	= \sum_{\ell=0}^{i} (-1)^{i+\ell} \binom{p-\ell}{p-i} 
			\binom{p-x}{\ell} \binom{N-p-x+\ell}{\ell}.\hspace{3cm}  \label{eq:eberlein}
\end{align}
Polynomial \eqref{eq:krawtchouk} is the Krawtchouk polynomial and \eqref{eq:eberlein} is the Eberlein polynomial (see Appendix \ref{app:A} for more details).  

The polynomials $v_{ij}(x,y)$ can be obtained by 
\begin{equation}\label{eq:defv}
p_{ij}(xy)=v_{ij}( \theta_{xy},\mu_{xy})
\end{equation}
 and they are $(1,0)$-compatible bivariate polynomials of degree $(i, j)$ where
\begin{subequations} \label{eq:tm}
    \begin{align}
        &\theta_{xy} = p_{10}(xy) = k(r-2)-x(r-1)  \,,\label{eq:theta}\\
        &\mu_{xy} = p_{01}(xy) = (r-1) ((k - x - y) (n-k-y)-y)\,. \label{eq:mu}
    \end{align}
\end{subequations}

\section{Bivariate $Q$-polynomial structure} \label{sec:biQpoly}

In this section, 
we show that when $n\geq 2k-1$, the non-binary Johnson scheme $J_r(k,n)$ is a bivariate $Q$-polynomial association scheme on $\cD^*=\{(a,b)\in \NN^2\, | \, a+b \leq k, \ b \leq k,n-k\}$ of type $(0,\frac12)$.  
We prove this by using Proposition \ref{pro:q1}.
Indeed, the dual eigenvalues are known \cite{TAG} and are given explicitly in terms of bivariate polynomials:
\begin{equation}
	q_{ij}(x,y) = \frac{\binom{n}{i}}{\binom{k}{i}}K_{i}(x,k-y,r-1)H_{j}(y,n-i,k-i), \label{eq:dualeigen}
\end{equation}
where $(i,j),(x,y) \in \{(a,b)\in \NN^2\, | \, a+b \leq k, \ b \leq k,n-k\}$ and, for $i=0,1,\dots,N$,
\begin{align}
	&H_i(x,N,p) =  \frac{\binom{N}{i}-\binom{N}{i-1}}{\binom{p}{x}\binom{N-p}{x}}E_x(i,N,p). \label{eq:hahn}
\end{align}
Polynomial \eqref{eq:hahn}  is the Hahn polynomial.
Let us emphasize that the following change of variables has been applied to the eigenvalues given in \cite{TAG}: $x \mapsto x+y, i\mapsto i+j, j\mapsto i$. With these variables, we have $\cD=\cD^\star$. The definition of the Hahn polynomial has also been changed: $H_i(N,p,x)\mapsto H_i(x,N,p)$. We will show that the polynomials $v^*_{ij}(x,y)$ defined by 
\begin{equation}\label{def:vs}
q_{ij}(xy)=v^\star_{ij}( \theta^\star_{xy},\mu^\star_{xy})
\end{equation}
 are $(0,\frac12)$-compatible bivariate polynomials of degree $(i, j)$ where
 \begin{subequations} \label{eq:tms}
 \begin{align}
  &\theta^\star_{xy}=q_{10}(xy)=\frac{n}{k}\Big((r-2)(k-y)-x(r-1)\Big)\,, \label{eq:thetas}\\
  &\mu^\star_{xy}=q_{01}(xy)=(n-1)\left(1-\frac{n}{k(n-k)}y\right)\,. \label{eq:mus}
 \end{align}
\end{subequations}

We first prove the following lemma allowing to express the Hahn polynomial $H_r(x,N,p)$ as a linear combination of $H_s(x,N-1,p-1)$'s, for some $s$ determined by $r$.
\begin{lem}\label{lem:hahnparach}
The following relation between Hahn polynomials holds 
\begin{equation*}
	H_r(x,N,p)=\frac Np \Bigg( \frac{p-r}{N-2r}H_r(x,N-1,p-1)+\frac{N-p-r+1}{N-2r+2}			H_{r-1}(x,N-1,p-1) \Bigg).
\end{equation*}
\end{lem}
\begin{proof}
Using definitions \eqref{eq:eberlein} and \eqref{eq:hahn} and after some manipulations of the binomial coefficients, one gets 
\begin{align}
	H_r(x,N,p)  
 &= \frac{\binom{N}{r}-\binom{N}{r-1}}{\binom{p}{x}\binom{N-p}{x}}
		\sum_{\ell=0}^{x} (-1)^{x+\ell} \binom{p-\ell}{p-x} 
			\binom{p-r}{\ell} \binom{N-p-r+\ell}{\ell}  \label{eq:Hrdire}\\
			&=\frac{1}{p 
			\binom{p-1}{x} \binom{N-p}{x}} 	
			\sum_{\ell=0}^{x} (-1)^{x+\ell} 
			(p-\ell)\frac{N-2r+1}{N-r+1}
			\binom{p-1-\ell}{p-1-x} 
			\binom{p-r}{\ell} \binom{N-p-r+\ell}{\ell}\binom{N}{r}. \label{eq:Hrsum}
\end{align} 
Let us focus on the the factor 
$D=(p-\ell)\frac{N-2r+1}{N-r+1}\binom{p-r}{\ell} \binom{N-p-r+\ell}{\ell}\binom{N}{r}$ in each summand and rewrite it as follows:
\begin{align*}  D
			&= 
			\Big(p-r-\ell +\frac{(N-p-r+\ell+1)r}{N-r+1}\Big)
			\binom{p-r}{\ell} \binom{N-p-r+\ell}{\ell}\binom{N}{r} \\
			&=
			(p-r-\ell)\frac{p-r}{p-r-\ell} \binom{p-1-r}{\ell} 
			\binom{N-p-r+\ell}{\ell}\binom{N}{r} \\
			&\quad+ \frac{(N-p-r+\ell+1)r}{N-r+1}\binom{p-r}{\ell} 
			\frac{N-p-(r-1)}{N-p-(r-1)+\ell}\binom{N-p-(r-1)+\ell}{\ell}\binom{N}{r}\\
			&=
			(p-r) \binom{p-1-r}{\ell} 
			\binom{N-p-r+\ell}{\ell}
			\frac{N}{N-r}\binom{N-1}{r} \\
			&\quad+ \frac{(N-p-r+1)r}{N-r+1}\binom{p-r}{\ell} 
			\binom{N-p-r+\ell+1}{\ell}
			\frac{N}{r}\binom{N-1}{r-1}\\
			&=
			(p-r) \binom{p-1-r}{\ell} 
			\binom{N-p-r+\ell}{\ell}
			\frac{N}{N-r} \frac{N-r}{N-2r}
			\bigg(\binom{N-1}{r}-\binom{N-1}{r-1} \bigg) \\
			&\quad+ \frac{(N-p-r+1)r}{N-r+1}\binom{p-r}{\ell} 
			\binom{N-p-r+\ell+1}{\ell}
			\frac{N}{r}\frac{N-r+1}{N-2r+2}
			\bigg(\binom{N-1}{r-1}-\binom{N-1}{r-2}\bigg)\\
			&=
			\frac{(p-r)N}{N-2r} \binom{p-1-r}{\ell} 
			\binom{N-p-r+\ell}{\ell}
			\bigg(\binom{N-1}{r}-\binom{N-1}{r-1} \bigg) \\
			&\quad+\frac{(N-p-r+1)N}{N-2r+2}
			\binom{p-r}{\ell} 
			\binom{N-p-r+\ell+1}{\ell}
			\bigg(\binom{N-1}{r-1}-\binom{N-1}{r-2}\bigg).
\end{align*}			
Putting the last expression in \eqref{eq:Hrsum} and comparing with \eqref{eq:Hrdire}, 
we get the desired result.  
\end{proof}

From this technical lemma, the recurrence relation of the dual eigenvalues $q_{ij}(x,y)$ can be obtained.
\begin{prop}\label{lem:recuq}
The dual eigenvalues $q_{ij}(x,y)$ of the non-binary Johnson scheme $J_r(k,n)$ 
satisfy 
\begin{align}
	&\theta^\star_{xy}\, q_{ij}(x,y)= \frac{n(r-3)i}{k} q_{ij}(x,y) 
	+  \frac{n(i+1)(k-i-j)}{k(n-i-2j)}q_{i+1,j}(x,y)
            +\frac{n(i+1)(n-k-j+1)}{k(n-i-2j+2)}q_{i+1,j-1}(x,y)\nonumber\\
   &   \quad     +\frac{n(n-j-k)(j+1)(r-2)}{k(n-i-2j)}q_{i-1,j+1}(x,y)+
	\frac{n(k-i-j+1)(n-i-j+2)(r-2)}{k(n-i-2j+2)}q_{i-1,j}(x,y),
 \label{eq:qrecu1}
\end{align} 
and 
\begin{equation}
\mu^\star_{xy} q_{ij}(x,y)= \wcA_{ij} \, q_{i,j+1}(x,y) + \wcB_{ij}\,  q_{i,j}(x,y) + \wcC_{ij}\,  q_{i,j-1}(x,y), 
 \label{eq:qrecu2}
\end{equation}
where $\theta^\star_{xy}$, $\mu^\star_{xy}$ are given by \eqref{eq:tms} and 
\begin{subequations}\label{eq:recuycoe}
\begin{align}
	&\wcA_{ij} = \frac{(n-1)n(j+k-n)(i+j-k)(j+1)}{k(n-k)(2j+i-n)(2j+i-n+1)}\,, \\
	&\wcB_{ij}= n-1 - \frac{(n-1)n\Big(j^2(n-i)-j(n-i)(n-i+1)+(k-i)(n-i+2)(n-k)\Big)}
	{k(n-k)(2j+i-n-2)(2j+i-n)}\,, \\
	&\wcC_{ij} = - \frac{(n-1)n(j+k-n-1)(i+j-k-1)(i+j-n-2)}{k(n-k)(2j+i-n-2)(2j+i-n-3)}\,.  
\end{align}
\end{subequations}
\end{prop}
\begin{proof}
To prove relation \eqref{eq:qrecu2}, replace $\mu^\star_{xy}$ by its value \eqref{eq:mus} to get
\begin{align*}
	\mu^\star_{xy} q_{ij}(x,y)
		&=(n-1)q_{ij}(x,y)+\frac{n(n-1)}{k(n-k)}
		  \frac{\binom{n}{i}}{\binom{k}{i}} 
		  K_i(x,k-y,r-1) (-y)H_j(y,n-i,k-i)\,.
\end{align*}
Using the recurrence relation \eqref{eq:hahnrecur} of $H_i$ given in Appendix \ref{app:A}, one proves \eqref{eq:qrecu2}.

For the other recurrence relation \eqref{eq:qrecu1}, 
one uses the recurrence relation for Krawtchouk polynomial (recalled in \eqref{eq:krawrecur})
to get 
\begin{align}
	\theta^\star_{x,y}q_{ij}(x,y)
	&=\frac{n(r-3)i}{k} q_{ij}(x,y)
	+ \frac{n(i+1)}{k}\frac{\binom{n}{i}}{\binom{k}{i}}K_{i+1}(x,k-y,r-1) H_j(y,n-i,k-i) \nonumber\\
	&\quad+ \frac{n(r-2)}{k} \frac{\binom{n}{i}}{\binom{k}{i}}(k-y-i+1)K_{i-1}(x,k-y,r-1) H_j(y,n-i,k-i)\,. \label{eq:q10prethree}
\end{align}
We now evaluate the last two summands in \eqref{eq:q10prethree}. Let us denote by $F$ and $G$ these two terms multiplied by $\frac{k}{n(i+1)}$ and $\frac{k}{n(r-2)}$, respectively. 
For both terms, 
 we need to change the Hahn polynomial parameters by replacing $i$ with $i-1$ or $i+1$ using Lemma \ref{lem:hahnparach}.
 The term $F$ can be rewritten as follows:
\begin{align}
	F &= \frac{\binom{n}{i}}{\binom{k}{i}}K_{i+1}(x,k-y,r-1) 
            \frac{n-i}{k-i}
            \bigg(
            \frac{k-i-j}{n-i-2j}H_j(y,n-(i+1),k-(i+1)) \nonumber\\
            & \quad +\frac{n-k-j+1}{n-i-2j+2}H_{j-1}(y,n-(i+1),k-(i+1))
            \bigg) \nonumber\\
            &= 
            \frac{k-i-j}{n-i-2j}q_{i+1,j}(x,y)
            +\frac{n-k-j+1}{n-i-2j+2}q_{i+1,j-1}(x,y). \label{eq:F}
\end{align}
To compute $G$, let us use the recurrence relation satisfied by the Hahn polynomials \eqref{eq:hahnrecur}:
\begin{align}
	G
	&=  \frac{\binom{n}{i}}{\binom{k}{i}}(k-i+1)K_{i-1}(x,k-y,r-1) H_j(y,n-i,k-i)
	  +\frac{\binom{n}{i}}{\binom{k}{i}}K_{i-1}(x,k-y,r-1) \nonumber\\
	 & \quad\Big(A_j(n-i,k-i)H_{j+1}(y,n-i,k-i) 
		   +B_j(n-i,k-i)H_j(y,n-i,k-i) \nonumber\\
	& \quad +C_j(n-i,k-i)H_{j-1}(y,n-i,k-i)\Big).   \label{eq:G1}
\end{align}
By letting $r=j, \ x=y, \ N=n-(i-1), \ p=k-(i-1)$ (resp. $r=j+1, \ x=y, \ N=n-(i-1), \ p=k-(i-1)$) in Lemma~\ref{lem:hahnparach}, 
we can express $H_{j-1}(y,n-i,k-i)$ (resp. $H_{j+1}(y,n-i,k-i)$) in \eqref{eq:G1}
in terms of $H_j(y,n-(i-1),k-(i-1))$ and $H_j(y,n-i,k-i)$ (resp. $H_{j+1}(y,n-(i-1),k-(i-1))$ and $H_j(y,n-i,k-i)$). 
Using these expressions,  replacing $A_j$, $B_j$, $C_j$ with \eqref{eq:hanrecur}, and simplifying, one gets 
\begin{align*}
	G
	& =  K_{i-1}(x,k-y,r-1)\frac{\binom{n}{i-1}}{\binom{k}{i-1}}
	\Big(\frac{(n-j-k)(j+1)}{n-i-2j}H_{j+1}(y,n-(i-1),k-(i-1))\nonumber\\
	& \quad+\frac{(k-i-j+1)(n-i-j+2)}{n-i-2j+2}H_{j}(y,n-(i-1),k-(i-1))\Big).
\end{align*}
Using definition  \eqref{eq:dualeigen}, it follows that
\begin{align}
	G
	&= \frac{(n-j-k)(j+1)}{n-i-2j}q_{i-1,j+1}(x,y)+
	\frac{(k-i-j+1)(n-i-j+2)}{n-i-2j+2}q_{i-1,j}(x,y). \label{eq:G}
\end{align}
Now combining equations \eqref{eq:q10prethree}, \eqref{eq:F}, and \eqref{eq:G}, 
we obtain 
the desired result. 
\end{proof}

\begin{thm}
Let $k$ and $n$ be positive integers such that $n\geq 2k-1$, 
then the non-binary Johnson scheme $J_r(k,n)$ is a bivariate $Q$-polynomial association scheme of type $(0,\frac12)$ on the region $\cD^*=\{(a,b)\in \NN^2\, | \, a+b \leq k, \ b \leq k,n-k\}$. 
\end{thm}
\begin{proof}
Let $v^*_{ij}(x,y)$ be defined by \eqref{def:vs}. 
Then by Proposition \ref{lem:recuq}, 
we have
\begin{multline*}
	xv^*_{ij}(x,y)= \frac{n(r-3)i}{k} v^*_{ij}(x,y) 
	+ \frac{n(i+1)(k-i-j)}{k(n-i-2j)}v^*_{i+1,j}(x,y)
            +\frac{n(i+1)(n-k-j+1)}{k(n-i-2j+2)}v^*_{i+1,j-1}(x,y)\\
           + \frac{n(n-j-k)(j+1)(r-2)}{k(n-i-2j)}v^*_{i-1,j+1}(x,y)+
	\frac{n(k-i-j+1)(n-i-j+2)(r-2)}{k(n-i-2j+2)}v^*_{i-1,j}(x,y), 
\end{multline*} 
and
\begin{equation*}
y v^*_{i,j}(x,y)=
\wcA_{ij}\, v^*_{i,j+1}(x,y) + \wcB_{ij}\, v^*_{i,j}(x,y) + \wcC_{ij}\, v^*_{i,j-1}(x,y),
\end{equation*}
where $\wcA_{ij}$, $\wcB_{ij}$, $\wcC_{ij}$ are as in \eqref{eq:recuycoe}. Comparing the terms which appear in these recurrence relations with the non-zero Krein parameters illustrated on Figure \ref{fig:recq10} and \ref{fig:recq01}, we conclude that $v^*_{ij}(x,y)$ is a $(0,\frac12)$-compatible bivariate polynomial of degree $(i,j)$. Moreover, the domain $\cD^\star$ is $(0,\frac12)$-compatible when $n-k \geq k-1$. 
Now the result follows from Proposition~\ref{pro:q1}.  
\end{proof}

\begin{rem}
    With respect to the definition of multivariate $Q$-polynomial association schemes in \cite{BKZZ}, 
    for this labeling of the principal idempotents $E_{mn}$ to be a bivariate $Q$-polynomial order, 
    we still need the condition $n-k\geq k-1$, even though $\cD^\star$ is not required to be $(\alpha,\beta)$-compatible. As if $n-k<k-1$, some monomial shows up in $v_{ij}(x,y)$ will be out of the region $\cD^\star$ (a requirement in the definitions in \cite{BKZZ}). 
\end{rem}

\begin{rem}
	In \cite{BCPVZ}, 
	the recurrence relation of the bivariate polynomials $v_{ij}(x,y)$ is obtained combinatorially, by calculating $A_{10}A_{ij}$ and $A_{01}A_{ij}$ directly, 
	with results as in \eqref{eq:rec10NBJ} and \eqref{eq:rec01NBJ}. 
	In fact, similar recurrences for $p_{ij}(x,y)$ as the ones for $q_{ij}$ in the Proposition \ref{lem:recuq} can be obtained by using the same proof idea  and the fact that $K_r(x,N,p)=K_r(x,N-1,p)+(p-1)K_{r-1}(x,N-1,p)$. 
    From the recurrences of $p_{ij}$ we can get the recurrences for $v_{ij}$. 
\end{rem}

\section{Bispectrality of $v_{ij}(x,y)$}\label{sec:bispectrality}

From relations \eqref{eq:rec10NBJ} and \eqref{eq:rec01NBJ}, 
we conclude easily that $v_{ij}(x,y)$ satisfy two recurrence relations:
\begin{subequations}
\begin{align}
	&\hspace{-0.4cm}x v_{ij}(x,y) = (k-i-j+1)(r-2) v_{i-1,j}(x,y)+(i+1)v_{i+1,j}(x,y) + \left( i(r-3)+j(r-2) \right) v_{ij}(x,y) \,,\label{eq:recp1} \\
	&\hspace{-0.4cm}y v_{ij}(x,y) =(k-i-j+1)(r-2)j v_{i-1,j}(x,y) + (k-i-j+1)(n-k-j+1)(r-1)v_{i,j-1}(x,y) \label{eq:recp2} \\ 
	&+ (i+1)j v_{i+1,j}(x,y) + (j+1)^2 v_{i,j+1}(x,y) + (i+1)(n-k-j+1)(r-1) v_{i+1,j-1}(x,y) \nonumber \\
	&+ (j+1)^2(r-2) v_{i-1,j+1}(x,y) +j\left( k-i-j+(r-2)i+(n-k-j)(r-1)\right)v_{i,j}(x,y). \nonumber
\end{align} 
\end{subequations}
In this section, we prove that $p_{ij}(x,y)$ satisfy also two difference relations which provides its bispectral property.
Despite the fact that the multivariate polynomials satisfying bispectral property have been already intensively studied \cite{Gri,Tra,GI,GI2,HR,CFR,GW}, it seems that the bispectrality of $v_{i,j}(x,y)$ was not established previously.
Let us emphasize that  $v_{i,j}(x,y)$ is a ``hybrid'' polynomial since it involves different types of univariate polynomials.

\subsection{Difference relation}

The result of Proposition \ref{lem:recuq} concerning the recurrence relation satisfied by $q_{ij}(x,y)$ allows to find a difference relation for $p_{ij}(x,y)$ using the Wilson duality \eqref{eq:qp}. 
\begin{prop}
Recalling that $p_{ij}(x,y)=v_{ij}( \theta_{xy},\mu_{xy})$ with $\theta_{xy}$ and $\mu_{xy}$ given by \eqref{eq:tm}, the following difference relations are satisfied:
\begin{subequations}
\begin{align}
&i p_{ij}(x,y)=\cP_1p_{ij}(x+1,y)+\cP_2p_{ij}(x+1,y-1) 
+\cP_3p_{ij}(x-1,y)
+\cP_4p_{ij}(x-1,y+1)\nonumber\\
&\qquad\qquad+\cP_5p_{ij}(x,y-1)
+\cP_6p_{ij}(x,y+1)
+\cP_7p_{ij}(x,y)\,,\label{eq:diffv1}
\\
&j p_{ij}(x,y)=B p_{ij}(x,y+1)-(B+D) p_{ij}(x,y) +D p_{ij}(x,y-1)\,,\label{eq:diffv2}
\end{align}
\end{subequations}
with 
\begin{subequations}
 \begin{align*}
	& B=\frac{ (y+x-k)(y+x-n-1)(y+k-n)}{(2y+x-n)(2y+x-n-1)}\,, &&D=-\frac{y(y+x-k-1)(y+k-n-1)}{(2y+x-n-1)(2y+x-n-2)}.
\end{align*}
\end{subequations}
and
\begin{subequations}
 \begin{align*}
	& \cP_1=-\frac{(r-2)(n-x-y+1)(k-x-y)}{(r-1)(n-x-2y+1)}\,, &&\cP_2 = \frac{y(r-2)(y+k-n-1)}{(r-1)(n-x-2y+1)}\,, \\
	&\cP_3 = -\frac{(k-x+1-y)x}{(r-1)(n-x-2y+1)}\,, &&
	\cP_4 = \frac{x(y+k-n)}{(r-1)(n-x-2y+1)}\,, \\
	&\cP_5 = -\frac{r-2}{r-1}D\,,\qquad\qquad \cP_6 = -\frac{r-2}{r-1}B\, , 
 &&\cP_7=-\cP_1-\cP_2-\cP_3-\cP_4-\cP_5-\cP_6.
\end{align*}
\end{subequations}
\end{prop}
\begin{proof}
Using the Wilson duality \eqref{eq:qp} and the recurrence relation for $q_{ij}(xy)$ \eqref{eq:qrecu2},
one gets 
	\begin{align*}
	&q_{01}(i,j) p_{ij}(x,y)
	= q_{01}(i,j) q_{xy}(i,j)\frac{k_{ij}}
	{m_{xy}} =\Big(
	\wcA_{xy} q_{x,y+1}(i,j)  + \wcB_{xy} q_{xy}(i,j) + \wcC_{xy} q_{x,y-1}(i,j)
	\Big)\frac{k_{ij}}
	{m_{xy}}   \\
	&=
	\wcA_{xy}\frac{m_{x,y+1}}{m_{xy}} p_{i,j}(x,y+1)  + \wcB_{xy} p_{i,j}(x,y) + \wcC_{xy}\frac{m_{x,y-1}}{m_{xy}} p_{i,j}(x,y-1).
	\end{align*}
By direct computation, one also gets
	\begin{equation*}
		%n_{ij}=p_{ij}(0,0)=(r-1)^j (r-2)^i 
		%\binom{k-j}{i} \binom{k}{j}  \binom{n-k}{j}\\
		m_{xy}=q_{xy}(0,0)=(r-2)^x 
		\binom{n}{x} \left(\binom{n-x}{y} - \binom{n-x}{y-1}\right).
	\end{equation*}
 From the explicit expressions \eqref{eq:recuycoe} and recalling that $q_{01}(i,j)=\Big((n-1)-\frac{n(n-1)}{k(n-k)}j\Big)$, 
	we obtain equation \eqref{eq:diffv2}. 
	Equation \eqref{eq:diffv1} is proven similarly. 
\end{proof}

\begin{rem}
The results can be derived also from the difference relation satisfied by the Krawtchouk and Eberlein polynomials directly. 
Making use of the recurrence relation of $q_{ij}(x,y)$ simplifies the computations, 
especially for \eqref{eq:diffv1}. 
\end{rem}

\subsection{Bispectral algebra}\label{subsec:bispalg}

For a given bispectral problem, it is natural to study the associated bispectral algebra. In \cite{Zhedanov}, such an algebra has been introduced for the bispectral problem associated to the Askey--Wilson polynomials and it leads to the definition of the eponym algebra. This algebra appears in many contexts now and, for example, it is used to characterize a Leonard pair \cite{TV},  
which is closely related to the univariate $P$- and $Q$-polynomial association schemes.

Let $W$ be the vector space spanned by $p_{i,j}(x,y)$ for $(i,j)\in \cD$ and for a given pair $(x,y)$.
Let $X^\star$ and $Y^\star$ be the matrices representing the action of the difference operators on $W$:
\begin{subequations}
\begin{align}
 & X^\star p_{i,j}(x,y) = i p_{i,j}(x,y),\\
  & Y^\star p_{i,j}(x,y) = j p_{i,j}(x,y),
\end{align}
\end{subequations}
and let $X$ and $Y$ be the matrices
representing the action of the recurrence operators on $W$:
\begin{subequations}
\begin{align}
 & X p_{i,j}(x,y) = (k-i-j+1)(r-2)p_{i-1,j}(x,y) + \left( i(r-3)+(j-k)(r-2) \right) p_{ij}(x,y) +(i+1)p_{i+1j}(x,y),\label{eq:Ds1}\\
  & Y p_{i,j}(x,y) = (k-i-j+1)(r-2)j p_{i-1,j}(x,y) + (k-i-j+1)(n-k-j+1)(r-1)p_{i,j-1}(x,y) \nonumber \\ 
	&+ (i+1)j p_{i+1,j}(x,y) + (j+1)^2 p_{i,j+1}(x,y) + (i+1)(n-k-j+1)(r-1) p_{i+1,j-1}(x,y) \nonumber \\
	&+ (j+1)^2(r-2) p_{i-1j+1}(x,y) +j\left( k-i-j+(r-2)i+(n-k-j)(r-1)\right)p_{ij}(x,y). \label{eq:Ds2}
\end{align}
\end{subequations}

Let us remark that 
\begin{equation}
[X,Y]=0\,,\qquad [X^\star,Y^\star]=0 \,.
\end{equation}
The second relation is obvious since the matrices are diagonal.
The first one is also easily proven since there exists another basis where $X$ and $Y$ are diagonal. 
By direct computation, one also gets 
\begin{equation}
[X,Y^\star]=0\,.
\end{equation}

The pair $(X,X^\star)$ satisfies the following relations
\begin{subequations}\label{eq:relX}
 \begin{align}
 &[X^\star,[X^\star,X]]=X -(r-3) X^\star -(r-2)(Y^\star-k),\\
 & [X,[X^\star,X]]=(r-3) X -(r-1)^2X^\star -(r-1)(r-2)(Y^\star-k).
\end{align}
\end{subequations}
It is a realization of the Lie algebra $gl_2$ with the three generators $X$, $X^\star$, $[X^\star,X]$ and the central element $Y^\star$.
Remarking that the recurrence and the difference relations associated to $X$ and $X^\star$ are closely related to the Krawtchouk part of $v_{ij}(x,y)$, it is not a surprise that the bispectral algebra is $gl_2$.

The pair $(Y,Y^\star)$ satisfies the following relations
\begin{subequations}\label{eq:relY}
 \begin{align}
 & [Y^\star,[Y^\star,Y]]=-2(1-r)(Y^\star)^2 +(n(1-r)-X)Y^\star + Y,\\
  & [Y,[Y^\star,Y]]=-2(r-1)\{Y,Y^\star\} + Y(X-n(1-r) )-2(k-n)(r-1)( X+k(r-1))\nonumber\\
  &\qquad -Y^\star\left( X^2+2(1-r)(n-2k-1)X+(r-1)^2(2n+(n-2k)^2)\right).
\end{align}
\end{subequations}
These relations are those of (a central extended version of) the Hahn algebra.

\subsection{Subconstituent algebra}

The matrices $A_{ij}$ form a commutative algebra, known as the Bose--Mesner algebra. 
For any association scheme, it is useful to introduce a more general algebra called the subconstituent algebra (or Terwilliger algebra) \cite{ter1,ter2,ter3}. We now do this for the non-binary Johnson scheme.

Let us fix an element $\bx$ of the set $S$ given in \eqref{eq:setXnbJ}. For $(i,j) \in \cD$ and $(m,n) \in \cD^\star$, we define the diagonal matrices $A_{mn}^\star=A_{mn}^\star(\bx)$ and $E_{ij}^\star=E_{ij}^\star(\bx)$ with entries
\begin{align}
(A^\star_{mn})_{\by \by }&=\bv (E_{mn})_{\bx\by},\\
(E^\star_{ij})_{\by \by } &=(A_{ij})_{\bx\by},
\end{align}
where $\bv = |S|$ here.
The matrices $A_{mn}^\star$  are the dual adjacency matrices (with respect to $\bx$) and they satisfy 
\begin{equation}
 A^\star_{ij}  A^\star_{k\ell} = \sum_{(m,n)\in\cD^\star} q_{ij,k\ell}^{mn} A^\star_{mn}.
\end{equation}
The matrices $E_{ij}^\star$  are the dual idempotents (with respect to $\bx$) and they satisfy
\begin{equation}
    E_{ij}^\star E_{mn}^\star = \delta_{ij,mn} E_{ij}^\star, \quad \sum_{(i,j)\in \cD} E_{ij}^\star = \II.
\end{equation}
We also have the relations
\begin{equation}
    A^\star_{mn} = \sum_{(i,j) \in \cD} q_{mn}(ij) E_{ij}^\star, \quad E^\star_{ij} =\frac{1}{\bv}\sum_{(m,n) \in \cD^\star} p_{ij}(mn) A^\star_{mn}. 
\end{equation}
The commutative algebra generalized by the matrices $A^\star_{ij}$ (or $E^\star_{ij}$) is called the dual Bose--Mesner algebra.
For a bivariate $Q$-polynomial association scheme, one gets 
\begin{equation}
 A^\star_{mn}=v^\star_{mn}( A^\star_{10}, A^\star_{01}).
\end{equation}

The algebra formed by $A_{ij}$ and $A^\star_{ij}$ is called the subconstituent algebra and is usually non-commutative. The following known result gives some general relations (adapted to the bivariate situation) which hold in the subconstituent algebra of an association scheme.
\begin{prop}\label{prop:EAE} \cite[Lemma 3.2]{ter1} For any fixed element $\bx \in S$ we have
    \begin{align}
        & E_{ij}^\star A_{mn} E_{rs}^\star = 0 \quad \Leftrightarrow \quad p_{ij,mn}^{rs} =0, \quad (i,j),(m,n),(r,s) \in \cD \\
        & E_{ij} A_{mn}^\star E_{rs} = 0 \quad \Leftrightarrow \quad q_{ij,mn}^{rs} =0, \quad (i,j),(m,n),(r,s) \in \cD^\star. 
    \end{align}
\end{prop}

Since the non-binary Johnson scheme is bivariate $P$- and $Q$-polynomial for $n\geq 2k-1$, its subconstituent algebra is generated by the four elements: $A_{10}$, $A_{01}$, $A^\star_{10}$ and $A^\star_{01}$. Note the useful relations
\begin{align}
    &A_{10} E_{ij} = E_{ij} A_{10}  = \theta_{ij} E_{ij}, \qquad A_{01} E_{ij} = E_{ij} A_{01} = \mu_{ij} E_{ij}, \\
    &A_{10}^\star E_{ij}^\star = E_{ij}^\star A_{10}^\star =  \theta_{ij}^\star E_{ij}^\star, \qquad A_{01}^\star E_{ij}^\star = E_{ij}^\star A_{01}^\star =  \mu_{ij}^\star E_{ij}^\star. 
\end{align}

We already know that $[A_{10},A_{01}]=[A^*_{10},A^*_{01}] = 0$ since the Bose--Mesner algebra and dual Bose--Mesner algebra are both commutative. We now provide some additional relations.
\begin{prop}
The following relations hold in the subconstituent algebra of the non-binary Johnson scheme $J_r(k,n)$ when $n\geq 2k-1$.
The elements $A^*_{01}$ and $A_{10}$ commute:
\begin{align}
    &[A^*_{01}, A_{10}]=0. \label{eq:relter1}
\end{align} 
The elements $A^*_{10}$ and $A_{10}$ satisfy the following relations:
\begin{subequations} \label{eq:DG1}
\begin{align}
    &[A^*_{10}, [A^*_{10}, [A^*_{10}, A_{10}]]] = 
 \left(\frac{n(r-1)}{k}\right)^2[A^*_{10}, A_{10}], \label{eq:relter2} \\
 &[A_{10}, [A_{10}, [A_{10}, A^*_{10}]]] = 
 (r-1)^2[A_{10},A^*_{10}]. \label{eq:relter4}
\end{align}
\end{subequations}
 The elements $A^*_{01}$ and $A_{01}$ satisfy (recall that  $A_{10}$ commutes with these two elements):
 \begin{subequations} \label{eq:TD1}
  \begin{align}
  &[A^*_{01}, [A^*_{01}, [A^*_{01}, A_{01}]]] = \left(\frac{n(n-1)}{k(n-k)}\right)^2[A^*_{01}, A_{01}], \label{eq:relter3} \\
 &[A_{01}, 2A_{01}A^*_{01}A_{01} - \{A_{01}^2,A^*_{01}\}  + 2(r-1)\{A_{01},A^*_{01}\} +(c_1+c_2A_{10}+A_{10}^2)A^*_{01}]=0, \label{eq:relter5}
\end{align}
\end{subequations}
where 
\begin{equation}
    c_1 = n(n+2)(r-1)^2 + k^2r^2 - 2k(r-1)(r(n+1)-2), \quad c_2 = 2 (kr-(n-1)(r-1)). \label{eq:coeffrelter5} 
\end{equation}
\end{prop}
\begin{proof}
    Using the fact that the sum of the idempotents is the identity, we can write  
    \begin{equation}
        [A^*_{01}, A_{10}] = \sum_{ij,ab \in \cD} E_{ij}^\star(A^*_{01}A_{10}-A_{10}A^*_{01})E_{ab}^\star = \sum_{ij,ab \in \cD} (\mu_{ij}^\star-\mu_{ab}^\star)E_{ij}^\star A_{10}E_{ab}^\star. \label{eq:relter1sum}
    \end{equation}
    Since the non-binary Johnson scheme is bivariate $P$-polynomial of type $(1,0)$, we have that $p_{ij,10}^{ab} =0$ if $|a-i|>1$ or $b \neq j$ (see Figure \ref{fig:recp10}). By Proposition \ref{prop:EAE}, it thus follows that $E_{ij}^\star A_{10}E_{ab}^\star =0$ if $b\neq j$. Equation \eqref{eq:mus} implies that $\mu_{xy}^\star$ does not depend on $x$, and hence $\mu_{ij}^\star-\mu_{ab}^\star=0$ if $b=j$. Therefore every term in the sum \eqref{eq:relter1sum} is zero, which proves \eqref{eq:relter1}. 

    To prove \eqref{eq:relter2}, we proceed similarly to write:
    \begin{equation}
        [A^*_{10}, [A^*_{10}, [A^*_{10}, A_{10}]]] = \sum_{ij,ab \in \cD} (\theta_{ij}^\star-\theta_{ab}^\star)^2E_{ij}^\star [A^*_{10}, A_{10}]E_{ab}^\star.
    \end{equation}
    Using equation \eqref{eq:thetas}, one gets
    \begin{equation}
        \theta_{ij}^\star-\theta_{ab}^\star = \frac{n}{k}((r-2)(b-j)+(r-1)(a-i)).
    \end{equation}
    Again, $E_{ij}^\star A_{10} E_{ab}^\star=0$ if $|a-i|>1$ or $b \neq j$. Moreover $\theta_{ij}^\star-\theta_{ab}^\star =0$ if $(i,j)=(a,b)$. Therefore $E_{ij}^\star[A^*_{10}, A_{10}]E_{ab}^\star = (\theta_{ij}^\star-\theta_{ab}^\star)E_{ij}^\star A_{10}E_{ab}^\star$ is possibly nonzero only if $a=i\pm 1$ and $b=j$. In any case, 
    \begin{equation}
        (\theta_{ij}^\star-\theta_{i\pm1,j}^\star)^2 = \left(\frac{n}{k}(r-1)\right)^2.
    \end{equation}
    The result \eqref{eq:relter2} then easily follows.

    Similarly for \eqref{eq:relter3},
    \begin{align}
        [A^*_{01}, [A^*_{01}, [A^*_{01}, A_{01}]]] &= \sum_{ij,ab \in \cD} (\mu_{ij}^\star-\mu_{ab}^\star)^2E_{ij}^\star [A^*_{01}, A_{01}]E_{ab}^\star.
    \end{align}
    Using equation \eqref{eq:mus}, one gets
    \begin{equation}
        \mu_{ij}^\star-\mu_{ab}^\star = \frac{n(n-1)}{k(n-k)}(b-j).
    \end{equation}
    For the non-binary Johnson scheme, we have $p_{ij,01}^{ab} =0$ if $|b-j|>1$ (see Figure \ref{fig:recp01}). If $b=j$, then $\mu_{ij}^\star-\mu_{ab}^\star =0$. Therefore, using Proposition \ref{prop:EAE}, $E_{ij}^\star[A^*_{01}, A_{01}]E_{ab}^\star = (\mu_{ij}^\star-\mu_{ab}^\star)E_{ij}^\star A_{01}E_{ab}^\star$ is possibly nonzero only if $b=j\pm 1$. In any case, 
    \begin{equation}
        (\mu_{ij}^\star-\mu_{a,j\pm1}^\star)^2 = \left(\frac{n(n-1)}{k(n-k)}\right)^2.
    \end{equation}
    Relation \eqref{eq:relter3} then follows.

    The computations are similar for \eqref{eq:relter4}. One must use the fact that the non-binary Johnson scheme is $Q$-polynomial of type $(0,\frac{1}{2})$ (for $n\geq 2k-1$), and hence $q_{ij,10}^{ab} =0$ if $|a-i|>1$ (see Figure \ref{fig:recq10}). One must also use equation \eqref{eq:theta} to find 
    \begin{equation}
        \theta_{ij}-\theta_{ab} = (r-1)(a-i).
    \end{equation}

    Finally, to prove \eqref{eq:relter5}, one uses again the idempotents to write the commutator on the LHS as
    \begin{equation}
        \sum_{ij,ab \in \cD} \left( 2 \mu_{ij}\mu_{ab} -(\mu_{ij}^2 + \mu_{ab}^2)+2(r-1)(\mu_{ij}+\mu_{ab}) +(c_1+c_2\theta_{ij}+\theta_{ij}^2) \right)(\mu_{ij}-\mu_{ab})E_{ij}A_{01}^\star E_{ab}. \label{eq:relter5sum}
    \end{equation}
    Since $q_{ij,01}^{ab} =0$ if $a\neq i$ or $|b-j|>1$, and since $\mu_{ij}-\mu_{ab} =0$ if $(i,j)=(a,b)$ (see Figure \ref{fig:recq01}), the only nonzero terms in the sum \eqref{eq:relter5sum} are those with $a=i$ and $b=j\pm 1$. In any case, it can be verified using the explicit expressions \eqref{eq:tm} for $\theta_{ij}$ and $\mu_{ij}$ that the first factor in the sum \eqref{eq:relter5sum} vanishes. This proves relation \eqref{eq:relter5}.  
\end{proof}

\begin{rem} Both relations \eqref{eq:DG1} between $A^*_{10}$ and $A_{10}$ are called the Dolan--Grady relations \cite{DG82}. These relations define the Onsager algebra \cite{Dav} which has been introduced to study the Ising model \cite{Ons}.

Both relations \eqref{eq:TD1} between $A^*_{01}$ and $A_{01}$ are known as the tridiagonal
 relations. Here we have a central extension of these relations 
 since $A_{10}$, which commutes with $A^*_{01}$ and $A_{01}$, appears in the coefficient of the relations.
 These relations have been introduced in \cite{ter3} in the context of univariate $P$- and $Q$-polynomial association schemes. It has been shown that these relations are satisfied for the subconstituent algebra of any univariate $P$- and $Q$-polynomial association scheme.
\end{rem}

\subsection{Relation between the subconstituent algebra and the bispectral algebra}

For any association scheme with $N$ classes, the subconstituent algebra (with respect to any vertex) has an irreducible module of dimension $N+1$ called the primary module \cite{ter1}. In the case of univariate $P$- and $Q$-polynomial association schemes, the representation of the subconstituent algebra on the primary module corresponds to the action of the bispectral operators on the associated polynomials, see for instance \cite{TerCourse}. 

The situation is analogous for bivariate $P$- and $Q$-polynomial association schemes. If we denote by $\hat{x}$ the vector with zero components everywhere except for the component labelled by $\bx \in S$ which is one, then a basis for the primary module is given by $\{A_{ij}\hat{x}\}_{(i,j)\in\cD}$. The representations of $A_{10}^\star$ and $A_{01}^\star$ on this basis are diagonal matrices, with the dual eigenvalues $\theta^\star_{ij}=q_{10}(ij)$ and $\mu^\star_{ij}=q_{01}(ij)$ as diagonal matrix elements, while the representations of $A_{10}$ and $A_{01}$ correspond to the action of the recurrence operators associated to the eigenvalues $\theta_{xy} = p_{10}(x,y)$ and $\mu_{xy} = p_{01}(x,y)$. Another basis for the primary module is given by the vectors $\{\bv E_{ij}\hat{x}\}_{(i,j)\in\cD^\star}$. On this basis, $A_{10}$ and $A_{01}$ are represented by diagonal matrices, with the eigenvalues $\theta_{xy}$ and $\mu_{xy}$ as diagonal elements, while the representations of $A_{10}^\star$ and $A_{01}^\star$ correspond to the action of the difference operators associated to the dual eigenvalues $\theta^\star_{ij}$ and $\mu^\star_{ij}$. Using the explicit expressions \eqref{eq:tm} and \eqref{eq:tms}, and comparing with the definitions of $X$, $Y$, $X^\star$, $Y^\star$ in Subsection \ref{subsec:bispalg}, one gets the following correspondence between the generators of the bispectral algebra and those of the subconstituent algebra on the primary module:   
\begin{subequations}\label{eq:XA}
\begin{align}
&X = A_{10}-k(r-2),\\
&Y = A_{01},\\
&X^\star = \frac{k(r-2)}{n(r-1)}\left( k + \frac{n-k}{n-1} A^*_{01} \right) - \frac{k}{n(r-1)} A^*_{10}, \\
&Y^\star = \frac{k(n-k)}{n}\left( 1- \frac{1}{n-1} A^*_{01}  \right).
\end{align}
\end{subequations}
\begin{rem}Relations \eqref{eq:XA} imply that $X$, $X^\star$, $Y$ and $Y^\star$ satisfy equivalent relations to the ones satisfied by $A_{10}$, $A_{01}$, $A^\star_{10}$, $A^\star_{01}$.
However, let us emphasize that the reverse is false: 
 $A_{10}$, $A_{01}$, $A^\star_{10}$, $A^\star_{01}$ do not satisfy the relations \eqref{eq:relX} and \eqref{eq:relY} verified by  $X$, $X^\star$, $Y$ and $Y^\star$. 
\end{rem}

\begin{rem}There exist other relations between $X$, $X^\star$, $Y$ and $Y^\star$ as well as between $A_{10}$, $A_{01}$, $A^\star_{10}$, $A^\star_{01}$. 
However, they seem quite complicated and are not displayed here.  
It would be very interesting to better understand these algebras
for any bivariate $P$- and $Q$- polynomial association scheme as it was done for the univariate case in \cite{ter3}. 
\end{rem}

\begin{rem}
The eigenvalues of association schemes based on attenuated spaces are given by the $q$-Krawtchouk polynomials and $q$-Hahn polynomials \cite{Kur}. 
Similar techniques as the ones used in this work could be used to show that these association schemes are also bivariate $Q$-polynomial.  
\end{rem}
\appendix

\section{Hypergeometric polynomials \label{app:A}}

We recall and prove in this appendix some properties 
of the Krawtchouk, Eberlein (dual Hahn), and Hahn polynomials.
They are given in terms of hypergeometric functions \cite{KLS}. 
Let us recall that the hypergeometric function ${_r}F_s$ is defined by the series 
\begin{equation*}\label{eq:hyperrFs}
	{_r}F_s\left( \begin{array}{c}
                                               a_1,\ldots, a_r\\
                                               b_1,\ldots, b_s
                                              \end{array}
 	 \Big| z \right)
	 =\sum_{\ell=0}^\infty \frac{(a_1)_{\ell}\cdots (a_r)_{\ell}}
	 {(b_1)_{\ell}\cdots (b_s)_{\ell}}
	 \frac{z^{\ell}}{\ell !},
\end{equation*} 
where $(a)_k=a(a+1)\cdots (a+k-1)$ is the Pochhammer symbol. 
The Krawtchouk, dual Hahn and Hahn polynomials are defined respectively by,
for $i=0,1,\ldots, n$,
\begin{align}\label{eq:hyperKraw}
	&\hat K_i(x;p,n)={_2}F_1\left( \begin{array}{c}
                                               -i,-x\\
                                               -n
                                              \end{array}
  \Big| \frac 1p \right), 
 \,,\\
& R_i(\lambda(x);\gamma,\delta,n)={_3}F_2\left( \begin{array}{c}
                                               -i,-x,x+\gamma+\delta+1\\
                                               \gamma+1,-n
                                              \end{array}
  \Big|1 \right)\,,\\
&  Q_i(x;\alpha,\beta,n)={_3}F_2\left( \begin{array}{c}
                                               -i, i+\alpha+\beta+1,-x\\
                                               \alpha+1,-n
                                              \end{array}
  \Big|1 \right)\,.
\end{align} 
The polynomials used in this paper are slightly modified in comparison to \cite{KLS}
\begin{subequations} 
\begin{align}
& K_i(x,N,p)= \binom{N}{i}(p-1)^i \hat K_i(x;(p-1)/p,N)\\
 &H_{i}(x,N,p)= \left(\, \binom{N}{i}- \binom{N}{i-1} \right) Q_i(x;\mu-N-1,-\mu-1,\mu),\\ 
 & E_i(x,N,p)=\binom{p}{i} \binom{N-p}{i} R_i(\lambda(x);\mu-N-1,-\mu-1,\mu),
\end{align}
\end{subequations}
with $\mu=\text{min}(p,N-p)$.
Using the recurrence relation of the Krawtchouk polynomials $\hat K_i$ given in \cite{KLS}, one gets that of $K_i$ 
\begin{equation}\label{eq:krawrecur}
	pxK_i(x,N,p)=-(i+1)K_{i+1}(x,N,p) + \big(i+(p-1)(N-i)\big)K_i(x,N,p)-
	(p-1)(N-i+1)K_{i-1}(x,N,p)\,.
\end{equation}

The recurrence relation for $Q_i(x;\alpha,\beta,n)$ is \cite{KLS}
\begin{equation}
	-xQ_i(x;\alpha,\beta,n)=A_iQ_{i+1}(x;\alpha,\beta,n)-(A_i+C_i)Q_i(x;\alpha,\beta,n)+C_iQ_{i-1}(x;\alpha,\beta,n), 
\end{equation}
where
    \begin{align}
	& A_i=\frac{(i+\alpha+\beta+1)(i+\alpha+1)(n-i)}
		{(2i+\alpha+\beta+1)(2i+\alpha+\beta+2)}\, ,\qquad C_i=\frac{i(i+\alpha+\beta+n+1)(i+\beta)}
		{(2i+\alpha+\beta)(2i+\alpha+\beta+1)}.
\end{align}
We deduce that of $H_i(x,N,p)$:
\begin{equation}\label{eq:hahnrecur}
	-x H_r(x,N,p)=
	A_r H_{r+1}(x,N,p) + B_r H_{r}(x,N,p) + C_r H_{r-1}(x,N,p),
\end{equation}
where 
\begin{subequations} \label{eq:hanrecur}
   \begin{align}
	& A_r(N,p)=\frac{(r-N+p)(p-r)(r+1)}{(2r-N)(N-2r-1)}, \\
	& B_r(N,p)=-\frac{r^2N-rN(N+1)+(2+N)(N-p)p}{(2r-N-2)(2r-N)}, \\
	& C_r(N,p)=-\frac{(r-N-1+p)(r-p-1)(N-r+2)}{(2r-N-2)(N-2r+3)}.
\end{align}
\end{subequations}

\vspace{2cm}
\noindent \textbf{Acknowledgements.}
NC thanks the CRM for its hospitality and is supported by the international research project AAPT of the CNRS and the ANR Project AHA ANR-18-CE40-0001. 
The research of LV is supported by a Discovery Grant from the Natural Sciences and Engineering Research Council (NSERC) of Canada.
MZ holds an Alexander-Graham-Bell scholarship from the Natural Sciences and Engineering Research Council of Canada (NSERC).


\begin{thebibliography}{9}

\bibitem{Bai} R.A. Bailey, \textsl{Association Schemes: Designed Experiments, Algebra and Combinatorics,} Cambridge Studies in Advanced Mathematics, Series Number 84, 2004.

\bibitem{BIT}
E. Bannai, E. Bannai, T. Ito and R. Tanaka,
\textsl{Algebraic Combinatorics,} De Gruyter, Berlin, Boston, 2021.

\bibitem{BI} E. Bannai and T. Ito,
\textsl{Algebraic combinatorics I,} Benjamin-Cummings, Menlo Park, 1984.

\bibitem{BKZZ} E. Bannai, H. Kurihara, D. Zhao and Y. Zhu,
\textsl{Multivariate P- and/or Q-polynomial association schemes,}
\texttt{arXiv:2305.00707}.

\bibitem{BCPVZ} 
P.-A. Bernard, N. Crampe, L. Poulain d'Andecy, L. Vinet and M. Zaimi, 
\textsl{Bivariate $P$-polynomial association schemes}, 
\href{https://arxiv.org/abs/2212.10824}{\texttt{arXiv:2212.10824}}, 2022.

\bibitem{BM} R.C. Bose and M.D. Mesner, \textsl{On linear associative algebras corresponding to association schemes of partially balanced designs,} 
Annals of Mathematical Statistics 30 (1959) 21--38.

\bibitem{BCN} A.E. Brouwer, A. Cohen and A. Neumaier, \textsl{Distance Regular-Graphs}, Springer-Verlag, Berlin, 1989.

\bibitem{CFR} N. Crampe, L. Frappat and E. Ragoucy,
\textsl{Representations of the rank two Racah algebra and orthogonal multivariate polynomials,}
Linear Algebra and its Applications 664 (2023) 165-215 and \texttt{arXiv:2206.01031}.

\bibitem{Dav} B. Davies,
\textsl{Onsager’s algebra and superintegrability,}
J. Phys.A 23 (1990) 2245--2261.

\bibitem{Del} P. Delsarte, 
\textsl{An algebraic approach to the association schemes of coding theory,} 
Philips Res. Rep. Suppl. No. 10 (1973).

\bibitem{DG82} L. Dolan and M. Grady,
\textsl{Conserved charges from self-duality,}
Phys. Rev.D 25 (1982) 1587--1604.

\bibitem{Dun} C.F. Dunkl,
\textsl{A Krawtchouk polynomial addition theorem and wreath products of symmetric
groups}, 
Indiana Univ. Math. J. 25 (1976) 335--358.

\bibitem{GI}
 J.S. Geronimo and P. Iliev,
 \textsl{Bispectrality of multivariable Racah-Wilson polynomials,} 
 Constr. Approx. 31 (2010) 417--457 and \href{https://arxiv.org/abs/0705.1469}{\texttt{arXiv:0705.1469}}.

\bibitem{GI2}
J.S. Geronimo and P. Iliev,
\textsl{Multivariable Askey--Wilson function and bispectrality,}
The Ramanujan journal 24 (2011) 273--287.

\bibitem{God2} C.D. Godsil,
\textsl{Algebraic Combinatorics,} New York: Chapman and Hall (1993).

 
\bibitem{Gri}
 R. Griffiths, \textsl{Orthogonal polynomials on the multinomial distribution,} 
 Austral. J. Statist. 13 (1971) 27--35.

\bibitem{GW}
W. Groenevelt and C. Wagenaar, 
\textsl{An Askey--Wilson algebra of rank 2}, \texttt{arXiv:2206.03986}.

\bibitem{HR} M.R. Hoare and M. Rahman,
\textsl{A probabilistic origin for a new class of bivariate polynomials,}
SIGMA 4 (2008) 89--106 \texttt{arXiv:0812.3879}.

\bibitem{ITT} T. Ito, K. Tanabe and P. Terwilliger,
\textsl{Some algebra related to $P$- and $Q$-polynomial association schemes}, 
Codes and Association Schemes (Piscataway NJ, 1999) 167–-192, DIMACS Ser. Discrete Math. Theoret. Comput. Sci. 56, Amer. Math. Soc., Providence RI 2001. \href{https://arxiv.org/abs/math/0406556}{\texttt{arXiv:math.CO/0406556}}.

 \bibitem{KLS} R. Koekoek,  P. Lesky. and R. Swarttouw, 
\textsl{Hypergeometric orthogonal polynomials and their
              {$q$}-analogues}, Springer-Verlag, Berlin, 2010.
\bibitem{Kur}H. Kurihara,
\textsl{Character Tables of Association Schemes Based on Attenuated Spaces,}
Ann. Comb. 17 (2013) 525--541.

\bibitem{Leo} D.A. Leonard,
\textsl{Orthogonal polynomials, duality and association schemes,}
SIAM J. Math. Anal. 13 (1982) 656--663. 

 \bibitem{Ons} L. Onsager,
 \textsl{Crystal Statistics. I. A Two-Dimensional Model with an Order-Disorder Transition,} Phys. Rev. 65 (1944) 117--149.
 
\bibitem{TAG} H. Tarnanen, M.J. Aaltonen and J.-M. Goethals,
\textsl{On the Nonbinary Johnson Scheme,}
Europ. J. Combinatorics 6 (1985) 279--285. 

\bibitem{ter1} P. Terwilliger, \textsl{The Subconstituent Algebra of an Association Scheme, (Part I)},
Journal of Algebraic Combinatorics 1 (1992) 363--388.

\bibitem{ter2} P. Terwilliger, \textsl{The Subconstituent Algebra of an Association Scheme, (Part II)},
Journal of Algebraic Combinatorics 2 (1993) 73--103.

\bibitem{ter3} P. Terwilliger, \textsl{The Subconstituent Algebra of an Association Scheme, (Part III)},
Journal of Algebraic Combinatorics 2 (1993) 177--210.

\bibitem{Ter01} P. Terwilliger, 
\textsl{Two linear transformations each tridiagonal with respect to an eigenbasis of the other}, 
Linear Algebra Appl. 330 (2001) 149--203, \href{https://arxiv.org/abs/math/0406555}{\texttt{arXiv:math/0406555}}.

\bibitem{TerCourse} P. Terwilliger,
\textsl{Distance-regular graphs, the subconstituent algebra, and the Q-polynomial property,}
\texttt{arXiv:2207.07747} 

\bibitem{TV} P. Terwilliger and R. Vidunas, 
\textsl{Leonard pairs and the Askey--Wilson relations,}
J. Algebra Appl. 3 (2004) 411--426 and \texttt{arXiv:math/0305356}.

\bibitem{Tra} M.V. Tratnik,
\textsl{Some multivariable orthogonal polynomials of the Askey tableau-discrete families,}
J. Math. Phys. 32 (1991) 2337--2342.

\bibitem{Zhedanov}
A.S. Zhedanov. \textsl{``Hidden symmetry'' of the Askey--Wilson polynomials,}
Theor. Math. Phys. \textbf{89} (1991) 1146--1157.

\bibitem{Zie2} P.-H. Zieschang, \textsl{The exchange condition for association schemes,} Israel Journal of Mathematics 151 (2006) 357--380.



\end{thebibliography}
\end{document}